\documentclass[11pt, a4paper]{article}
\usepackage{amsmath,amssymb,latexsym,color,enumerate,amsthm}
\usepackage{indentfirst}
\usepackage{authblk}
\usepackage[mathscr]{eucal}
\usepackage[colorlinks,
            linkcolor=blue,
            anchorcolor=green,
            citecolor=magenta
           ]{hyperref}
\usepackage{setspace}
\newtheorem{thm}{Theorem}[section]
\newtheorem{prop}[thm]{Proposition}
\newtheorem{lemma}[thm]{Lemma}
\newtheorem{cor}[thm]{Corollary}

\newtheorem{example}{Example}[section]
\newtheorem{defin}{Definition}[section]

\def\q{\hfill\rule{1ex}{1ex}}

\usepackage{CJK}

\begin{document}
%\begin{CJK*}{GBK}{song}

\renewcommand{\baselinestretch}{1.3}
%%%%%%%%%%%%%%%%%%%%%%%%%%%%%%%%%%%%%%%%%%%%%%%%%%%%%%%%%%%%%%%%%%%%%%%%%%%%%%%%%%%%%%%%
%%%%%%%%%%%%%%%%%%%%%%%%%%%%%%%%%%%%%%%%%%%%%%%%%%%%%%%%%%%%%%%%%%%%%%%%%%%%%%%%%%%%%%%%

\title{\bf Simplices in $t$-intersecting families for vector spaces}

\author[1]{Haixiang Zhang\thanks{E-mail: \texttt{zhang-hx22@mails.tsinghua.edu.cn}}}
\author[2]{Mengyu Cao\thanks{Corresponding author. E-mail: \texttt{myucao@ruc.edu.cn}}}
%\author[3]{Jiaqi Liao\thanks{E-mail: \texttt{975497560@qq.com}}}
\author[1]{Mei Lu\thanks{E-mail: \texttt{lumei@tsinghua.edu.cn}}}
\author[1]{Jiaying Song\thanks{E-mail: \texttt{jy-song21@mails.tsinghua.edu.cn}}}

\affil[1]{\small Department of Mathematical Sciences, Tsinghua University, Beijing 100084, China}
\affil[2]{\small Institute for Mathematical Sciences, Renmin University of China, Beijing 100086, China}
%\affil[3]{\small Academy of Mathematics and Systems Science, Chinese Academy of Sciences, Beijing 100190, China}

\date{}
\maketitle

\begin{abstract}
    Let $V$ be an $n$-dimensional vector space over the finite field $\mathbb{F}_q$ and ${V\brack k}$ denote the family of all $k$-dimensional subspaces of $V$. A family $\mathcal{F}\subseteq {V\brack k}$ is called $k$-uniform $r$-wise $t$-intersecting if for any $F_1, F_2, \dots, F_r \in \mathcal{F}$, we have $\dim\left(\bigcap_{i=1}^r F_i \right) \geq t$. An $r$-wise $t$-intersecting family $\{X_1, X_2, \dots, X_{r+1}\}$ is called a $(r+1,t)$-simplex if $\dim\left(\bigcap_{i=1}^{r+1} X_i \right) < t$, denoted by $\Delta_{r+1,t}$. Notice that it is usually called triangle when $r=2$ and $t=1$. For $k \geq t \geq 1$, $r \geq 2$ and $n \geq 3kr^2 + 3krt$, we prove that the maximal number of $\Delta_{r+1,t}$ in a $k$-uniform $r$-wise $t$-intersecting subspace family of $V$ is at most $n_{t+r,k}$, and we describe all the extreme families. Furthermore, we have the extremal structure of $k$-uniform intersecting families maximizing the number of triangles for $n\geq 2k+9$ as a corollary.


\end{abstract}

\section{Introduction}
Throughout this paper, we fix three integers $r$, $t$ and $k$, where $r \geq 2$, $t \geq 1$ and $t \leq k - r$. Let $V$ be an $n$-dimensional vector space over the finite field $\mathbb{F}_q$ (with $q$ being a prime power), and let ${V \brack k}$ denote the set of all $k$-dimensional subspaces of $V$. From now on, we will abbreviate ``$k$-dimensional subspace" as ``$k$-subspace". The function \(\text{span}(V_1,V_2,\ldots,V_s,e_1,e_2,\ldots,e_{\ell})\) denotes the minimal subspaces of $V$ contains all the subspaces \(V_1,V_2,\ldots,V_s\) and vectors \(e_1,e_2,\ldots,e_{\ell}\). Sometimes we also abbreviate $\text{span}(\cdot )$ as $<\cdot >$. Recall that for any positive integers $a$ and $b$, the \emph{Gaussian binomial coefficient} is defined by
$
{a \brack b} = \prod_{0 \leq i < b} \frac{q^{a-i} - 1}{q^{b-i} - 1}.
$
Additionally, we set ${a \brack 0} = 1$, and ${a \brack c} = 0$ if $c$ is a negative integer. It is also worth noting that the size of ${V \brack k}$ is equal to ${n \brack k}$.

A family $\mathcal{F} \subseteq {V \brack k}$ is called \emph{$k$-uniform $r$-wise $t$-intersecting} if for any $F_1, F_2, \dots, F_r \in \mathcal{F}$, we have $\dim\left( \bigcap_{i=1}^r F_i \right) \geq t$. The family is assumed to be $k$-uniform throughout this paper, unless stated otherwise. Note that the term ``$2$-wise $t$-intersecting" refers to the classical concept of a ``$t$-intersecting" family. For convenience, we will simply refer to a $1$-intersecting family as an \emph{intersecting family}. A $r$-wise $t$-intersecting family $\mathcal{F}$ is \emph{trivial} if there exists $X \in {V \brack t}$ such that $X \subseteq F$ for all $F \in \mathcal{F}$, and \emph{non-trivial} otherwise. A $r$-wise $t$-intersecting family $\mathcal{F}$ is \emph{maximal} if adding another $k$-subspace to $\mathcal{F}$ would violate the $r$-wise $t$-intersecting property. A family $\mathcal{F}_1, \mathcal{F}_2, \dots, \mathcal{F}_m$ is called \emph{$m$-cross $s$-intersecting} if for any $F_i \in \mathcal{F}_i$ ($1 \leq i \leq m$), we have
$
\dim \left( \bigcap_{i=1}^m F_i \right) \geq s.
$

The study of uniform intersection families has a long history, with increasing interest in recent years, particularly following the Erd\H{o}s-Ko-Rado (EKR) theorem \cite{Erdos-Ko-Rado-1961-313}, which determined the maximum size of an $t$-intersecting family of ${[n]\choose k}$ for $n>n_0(k,t)$. It is known that the smallest possible such function $n_0(k, t)$ is $(t +1)(k-t +1).$ This was partially proved by Frankl \cite{Frankl-1978} and subsequently determined by Wilson \cite{Wilson-1984}. In \cite{Frankl-1978}, Frankl also made a conjecture on the maximum size of a $t$-intersecting family of ${[n]\choose k}$ for any positive integers $t,k$ and $n$. This conjecture was partially proved by Frankl and F\"{u}redi   \cite{Frankl--Furedi-1991} and completely settled by Ahlswede and Khachatrian  \cite{Ahlswede-Khachatrian-1997}. The EKR theorem has since been generalized in many directions. For example, the structure of $t$-intersecting families of ${V \brack k}$ with maximum size has been fully determined \cite{Deza-Frankl-1983,PL,PR,Hsieh-1975-1,WN,Tanaka-2006-903}, which are known as the Erd\H{o}s-Ko-Rado theorem for vector spaces. The problem of maximizing the size of $r$-wise $t$-intersecting families for vector spaces has also been well-studied. In \cite{Chowdhury-2010}, Chowdhury and Patk\'{o}s used shadows in vector spaces to determine the structure of extremal $r$-wise $t$-intersecting families of ${V \brack k}$. For other works on $r$-wise $t$-intersecting families, we refer to \cite{cao2023r,O-V-2021,Xiao-2018}.

In a $k$-uniform intersecting family $\mathcal{F} \subseteq {V \brack k}$, a triple $\{F_1, F_2, F_3\}$ in $\mathcal{F}$ is called a \emph{triangle} if $\dim(F_1 \cap F_2 \cap F_3) = 0$. Moreover,  an $r$-wise $t$-intersecting family $\{F_1, F_2, \dots, F_{r+1}\}$ is called a \emph{$(r+1,t)$-simplex} if $\dim\left(\bigcap_{i=1}^{r+1} F_i \right) < t$, denoted by $\Delta_{r+1,t}$. The number of $(r+1)$-subsets of $\mathcal{F}$ that form a $\Delta_{r+1,t}$ is denoted by $\mathcal{N}(\Delta_{r+1,t}, \mathcal{F})$.

In \cite{NP2022}, Nagy and Patk\'{o}s defined the triangle for finite sets. A collection $\mathcal{T}$ of $(r+1)$ sets as an $(r+1)$-\emph{triangle} if for every $T_1, T_2, \dots, T_r \in \mathcal{T}$, we have $T_1 \cap \cdots \cap T_r \neq \emptyset$, but $\bigcap_{T \in \mathcal{T}} T = \emptyset$. They discussed the structure of the $r$-wise intersecting family $\mathcal{F}$ that contains the maximum number of $(r+1)$-triangles. Liao, Cao, and Lu \cite{Liao-2023} generalized the definition of the $(r+1)$-triangle to the $(r+1,t)$-triangle for $r$-wise $t$-intersecting families and determined the extremal structure containing the maximum number of $(r+1,t)$-triangles. Note that $r$-triangle is also referred to as $r$-simplex in \cite{FT2016}, and we use this definition in this paper. Our result can be viewed as a vector space version of the result in \cite{Liao-2023}.

The main tool used in this paper is the \emph{$t$-covering number}. For any $\mathcal{F} \subseteq {V \brack k}$ (not necessarily $t$-intersecting), a subspace $T$ of $V$ is called a \emph{$t$-cover} of $\mathcal{F}$ if $\dim(T \cap F) \geq t$ for all $F \in \mathcal{F}$, and the \emph{$t$-covering number} $\tau_t(\mathcal{F})$ is the minimum dimension of a $t$-cover. It is clear that a $t$-intersecting family $\mathcal{F}$ is trivial if and only if $\tau_t(\mathcal{F}) = t$. Denote by $\mathcal{C}_t(\mathcal{F})$ the family of all $t$-covers of $\mathcal{F}$ with size $\tau_t(\mathcal{F})$. The $t$-covering number is a powerful tool for studying intersecting families. In \cite{AB}, using the $1$-covering number, Blokhuis, Brouwer, Chowdhury, Frankl, Mussche, Patk\'{o}s and Sz\H{o}nyi obtained a vector space version of the Hilton-Milner theorem. In \cite{Cao-vec}, the authors used the $t$-covering number to describe the structure of maximal non-trivial $t$-intersecting families of ${V \brack k}$ with large size. In \cite{Liao-2023} and \cite{NP2022}, the $t$-covering number was used to study the triangles in $r$-wise $t$-intersecting families.

Define the following families:
\[
\mathcal{F}_{X,k} = \left\{ F \in {V \brack k} : \dim(F \cap X) \geq \dim(X) - 1 \right\},
\]
\[
\mathcal{F}^*_{X,k} = \left\{ F \in {V \brack k} : \dim(F \cap X) = \dim(X) - 1 \right\}.
\]
We also define
\[
n_{t+r,k} = \mathcal{N}(\Delta_{r+1,t}, \mathcal{F}_{X,k}),
\]
where $X$ is a $(t+r)$-subspace of $V$.

The main results of this paper are as follows.
\begin{thm}\label{th}
    Let $\mathcal{F} \subseteq {V \brack k}$ be an $r$-wise $t$-intersecting family with $t \geq 1$, $r \geq 2$, $k \geq r+t-1$ and $n \geq 3kr^2 + 3krt$. Then
    \[
    \mathcal{N}(\Delta_{r+1,t}, \mathcal{F}) \leq n_{t+r,k},
    \]
    with equality holding if and only if $\mathcal{F}^*_{X,k} \subseteq \mathcal{F} \subseteq \mathcal{F}_{X,k}$ for some $(r+t)$-subspace $X$ of $V$.
\end{thm}

Notice that any $r$-wise $t$-intersecting family with $k < r+t-1$ is trivial (otherwise, by Lemma \ref{rt}, we have the family is $(r+t-2)$-intersecting. It indicates that the family is a single $(r+t-2)$-subspace, a contradiction) and contains no simplices. Thus we only need to consider the case $k\geq r+t-1$. We always assume $t \geq 1$, $r \geq 2$ and $k \geq r+t-1$ in the following sections.

\begin{cor}\label{cor}
    Let $\mathcal{F} \subseteq {V \brack k}$ be an intersecting family with $k\geq 1$ and $n \geq 2k+9$. Then
    \[
    \mathcal{N}(\Delta, \mathcal{F}) \leq n_{3,k},
    \]
    with equality holding if and only if $\mathcal{F}^*_{X,k} \subseteq \mathcal{F} \subseteq \mathcal{F}_{X,k}$ for some $3$-subspace $X$ of $V$.
\end{cor}

The paper is organized as follows.  In Section 2, we will provide estimates for the extremal values and prove that the $t$-covering number of the extremal structure is exactly $(t+1)$. In Section 3, we will prove our main results.

\section{Preliminary}

Given a subspace $X\subseteq V$, we say that subspaces $F \subseteq V$ and $G \subseteq V$ are equivalent with respect to $X$, if $F \cap X = G \cap X$ and $F + X = G + X$. For any given $X$, there is a corresponding equivalence relation. For an ordered subspace pair $(X,Y)$ of $V$ such that $X\oplus Y=V$ and $F \subseteq V$, notice that there is a unique subspace $F(Y) \subseteq Y$ such that $(F \cap X) + F(Y)$ is equivalent to $F$ with respect to $X$. We call $(F \cap X, F(Y))$ the projection of $F$ onto $(X, Y)$. A natural question is how many elements are in an equivalence class. The following lemma is helpful for calculation.

\begin{lemma}{\rm(\cite[Lemma~2.1]{wang2010association})}\label{1}
    Fix $A \subseteq V$ with $\dim(A) = i$. Then the number of $j$-subspaces $B \in {V\brack j}$ such that $\dim(A \cap B) = \ell$ is given by \[
    \left|\left\{ B \in {V\brack j} : \dim(A \cap B) = \ell \right\}\right| = q^{(i - \ell)(j - \ell)} {i \brack \ell} {n - i \brack j - \ell}.
    \]
\end{lemma}

The following  formulas are commonly used  for computation in vector spaces.

\begin{prop}\label{eq}
Let $m$ and $i$ be positive integers with $i\leq m.$ Then the following hold:
\begin{itemize}
\item[{\rm (i)}] ${m\brack i}={m-1\brack i-1}+q^i{m-1\brack i}$ and ${m\brack i}=\frac{q^m-1}{q^i-1}\cdot{m-1\brack i-1}$;
\item[{\rm (ii)}] $q^{m-i}<\frac{q^m-1}{q^i-1}<q^{m-i+1}$ and $q^{i-m-1}<\frac{q^i-1}{q^m-1}<q^{i-m}$ if $i <m$;
\item[{\rm(iii)}] $q^{i(m-i)}\leq{m\brack i}< q^{i(m-i+1)}$ and $q^{i(m-i)}<{m\brack i}$ if $i <m$;
\item[{\rm (iv)}] $\frac{q^m-1}{q^i-1}<2q^{m-i}$.
\end{itemize}
\end{prop}

Next, we provide an estimation for the maximum number of $(r+1,t)$-simplices.

\begin{lemma}\label{n r+1,t}
    For \( n > 2k \), we have
    \[
    n_{t+r,k} \geq \frac{1}{2(r+1)!} q^{r(k-r-t+1) + \frac{r(r+1)}{2}} {n-r-t \brack k-r-t+1}^{r+1} \prod_{i=0}^{r} {t+i \brack t+i-1}.
    \]
    Moreover,
    \begin{align*}\label{eq1}
        \log_qn_{t+r,k}\geq (r+1)(k-r-t+1)n-(k+1)(r+1)(k-r-t+1)+r(k+1)-r^2.
    \end{align*}
\end{lemma}

\begin{proof}
    Let \( X \) be the \( (t+r) \)-subspace given in \( \mathcal{F}_{X,k} \), and let \( Y \) be another space such that \( V = X \oplus Y \).

    \noindent\textbf{Step 1:} Consider the subspaces \( F_1, F_2, \ldots, F_{r+1} \in \mathcal{F}^*_{X,k}\), whose projections onto \( (X, Y) \) are \( (F_1 \cap X, F_1(Y)), (F_2 \cap X, F_2(Y)), \ldots, (F_{r+1} \cap X, F_{r+1}(Y)) \), respectively. We claim that $ \{F_1, F_2, \ldots, $ $F_{r+1}\} $ forms an \( (r+1,t) \)-simplex if \( \dim \left( \bigcap_{i=1}^{r+1} (F_i \cap X) \right) = t-1 \) and \( \bigcap_{i=1}^{r+1} F_i(Y) = \mathbf{0} \).

    First, we verify that it forms an \( r \)-wise \( t \)-intersecting family. By induction on \( i \) and using the dimension formula, we can show that the dimension of the intersection of any \( i \) spaces in \( \{F_1, F_2, \ldots, F_{r+1}\} \) is at least \( t + r - i \). When $i=r$, we have the dimension of the intersection of any \( r \) spaces in \( \{F_1, F_2, \ldots, F_{r+1}\} \) is at least \( t \).

    Next, we verify that \( \dim \left( \bigcap_{i=1}^{r+1} F_i \right) < t \). Indeed, \( \bigcap_{i=1}^{r+1} F_i \subseteq X \), otherwise \( \bigcap_{i=1}^{r+1} F_i(Y) \neq \mathbf{0} \). Therefore, \( \dim \left( \bigcap_{i=1}^{r+1} F_i \right) = \dim \left( \bigcap_{i=1}^{r+1} (F_i \cap X) \right) = t - 1 \).

    \noindent\textbf{Step 2:} Let \( f(t, r) \) be the number of \( (r+1) \)-sequences \( (G_1, G_2, \ldots, G_{r+1}) \) such that \( G_i \in {X\brack t+r-1} \) and \( \dim \left( \bigcap_{i=1}^{r+1} G_i \right) = t - 1 \). We estimate \( f(t, r) \) by induction.

    For \( r = 1 \), fix any given \( G_1 \), by Lemma \ref{1}, there are \( q {t \brack t-1} {1 \brack 1} \) distinct \(  G_2 \). Hence, the number of $(G_1,G_2)$ pairs is \( q {t \brack t-1} {t+1 \brack t} \). For any general \( r \), we first fix \( G_{r+1} \), and there are \( {t+r \brack t+r-1} \) choices for it. Arbitrarily choose \( e \in X \setminus G_{r+1} \), so that \( X = G_{r+1} \oplus \langle e \rangle \). By induction, there are \( f(t, r-1) \) different \( r \)-sequences \( (G_1 \cap G_{r+1}, G_2 \cap G_{r+1}, \ldots, G_r \cap G_{r+1}) \), where each \( G_i \cap G_{r+1} \in {G_{r+1}\brack t+r-2} \), and \( \dim \left( \bigcap_{i=1}^{r} (G_i \cap G_{r+1}) \right) = t - 1 \). Given \( G_1 \cap G_{r+1} \), by Lemma \ref{1}, there are \( q \) distinct choices for \( G_1 \). Therefore, for each given \( (G_1 \cap G_{r+1}, G_2 \cap G_{r+1}, \ldots, G_r \cap G_{r+1}) \), there are \( q^r \) distinct sequences \( (G_1, G_2, \ldots, G_r) \). Finally, we obtain
    \begin{align*}
    f(t, r) &\geq {t+r \brack t+r-1} q^r f(t, r-1)\geq q^{\frac{r(r+1)}{2}} \prod_{i=0}^{r} {t+i \brack t+i-1}.
    \end{align*}

    \noindent\textbf{Step 3:} Let \( g(t, r) \) be the number of \( (r+1) \)-sequences \( (H_1, H_2, \ldots, H_{r+1}) \) such that \( H_i \in {Y\brack k-t-r+1} \), and \( \dim \left( \bigcap_{i=1}^{r+1} H_i \right) = 0 \). Using inclusion-exclusion, we have
    \begin{align*}
        g(t, r) &\geq {n - r - t \brack k - r - t + 1}^{r+1} - {n - r - t \brack 1} {n - r - t - 1 \brack k - r - t}^{r+1}.\\
    \end{align*}
    For $n\geq 2k$, by Proposition \ref{eq}, we have
    \begin{align*}
        {n - r - t \brack 1} {n - r - t - 1 \brack k - r - t}^{r+1}&<q^{(n-r-t)}\left(\frac{q^{k-r-t+1}-1}{q^{n-r-t}-1}{n - r - t  \brack k - r - t+1}\right)^{r+1}\\
        &\leq q^{(n-r-t)}\left(q^{-(n-k-1)}{n - r - t  \brack k - r - t+1}\right)^{r+1}\\
                 &\leq q^{-1}{n - r - t  \brack k - r - t+1}^{r+1}
        \leq \frac{1}{2}{n - r - t \brack k - r - t + 1}^{r+1}.
    \end{align*}
    Thus,
    \begin{align*}
        g(t, r) &\geq \frac{1}{2}{n - r - t \brack k - r - t + 1}^{r+1}.
    \end{align*}
    \noindent\textbf{Step 4:} Given a fixed \( (r+1) \)-sequence \( (G_1, G_2, \ldots, G_{r+1}) \) such that \( G_i \in {X\brack t+r-1} \) and \( \dim \left( \bigcap_{i=1}^{r+1} G_i \right) = t - 1 \), and a fixed \( (r+1) \)-sequence \( (H_1, H_2, \ldots, H_{r+1}) \) such that \( H_i \in {Y\brack k - t - r + 1} \) and \( \dim \left( \bigcap_{i=1}^{r+1} H_i \right) = 0 \). Let \( h(t, r) \) be the number of \( (r+1) \)-sequences \( (F_1, F_2, \) \(\ldots, F_{r+1})\) in \(\mathcal{F}^*_{X,k} \) such that the projection of $F_i$ onto \( (X, Y) \) is \( (G_i, H_i) \).

    From Lemma \ref{1}, we get the number of  $F\in {V\brack k}$ with $\dim(F\cap X)=t+r-1$ is
    \begin{align*}
        q^{k - r - t + 1}{r+t\brack r+t-1}{n - r - t\brack k-r-t+1}.
    \end{align*}
    All the projections of $F$ onto $(X,Y)$ are uniformly distributed over all $(A, B)$, for any $A \in {X\brack t+r-1}$ and $B\in {Y\brack k - t - r + 1}$. Thus, for any $1\leq i \leq r+1$, the number of $F_i$ whose projection onto $(X,Y)$ is $(G_i,H_i)$ is
    \begin{align*}
        q^{k - r - t + 1}{r+t\brack r+t-1}{n - r - t\brack k-r-t+1}/\left({r+t\brack r+t-1}{n - r - t\brack k-r-t+1}\right)=q^{k - r - t + 1}.
    \end{align*} Therefore, we immediately get
    \(
    h(t, r) = \left( q^{k - r - t+1} \right)^r.
    \)

    \noindent\textbf{Step 5:} We are now ready to estimate \( n_{t+r, k} \). By Steps 1 and 4, \( (r+1) \)-sequences \(( F_1, F_2, \ldots, F_{r+1})\) in \( \mathcal{F}^*_{X,k} \) such that the projection of $F_i$ onto \( (X, Y) \) is \( (G_i, H_i) \), form an \( (r+1,t) \)-simplex. By summing up all possible \( (r+1) \)-sequences \( (G_1, G_2, \ldots, G_{r+1}) \) and \( (r+1) \)-sequences \( (H_1, H_2, \ldots, H_{r+1}) \), we count each \( (r+1,t) \)-simplex \( (r+1)! \) times. Hence,
    \begin{align*}
    n_{t+r, k} &\geq \frac{f(t, r) g(t, r) h(t, r)}{(r+1)!} \\
    &\geq \frac{1}{2(r+1)!} q^{r(k-r-t+1) + \frac{r(r+1)}{2}} {n - r - t \brack k - r - t + 1}^{r+1} \prod_{i=0}^{r} {t+i \brack t+i-1}.
    \end{align*}
    By Proposition \ref{eq}, and $\log_q(2(r+1)!)\leq r^2$ we have
    \begin{align*}
        \log_qn_{t+r,k}&\geq -\log_q(2(r+1)!)+ r(k-r-t+1) \\
        &\quad+ \frac{r(r+1)}{2}+(r+1)(k-r-t+1)(n-k-1)+\sum_{i=0}^r(t+i-1)\\
        &\geq (r+1)(k-r-t+1)n-(k+1)(r+1)(k-r-t+1)+r(k+1)-r^2.
    \end{align*}

We get the lemma holds.
\end{proof}

\begin{lemma}\label{rt}
    Let \( \mathcal{F} \subseteq {V \brack k} \) be an \( r \)-wise \( t \)-intersecting family with \( \tau_t(\mathcal{F}) = \ell \) for $r\geq 2$. Then \( \mathcal{F} \) is \( ((r-2)(\ell-t) + t) \)-intersecting.
\end{lemma}

\begin{proof}
    Let \( F_1, F_2 \in \mathcal{F} \) be chosen randomly, and suppose that \( \dim(F_1 \cap F_2) = s \). Since \( \mathcal{F} \) is \( r \)-wise \( t \)-intersecting, the dimension of the intersection of \( F_1 \cap F_2 \) with any other space in \( \mathcal{F} \) is at least \( t \). This implies that \( \dim(F_1 \cap F_2) \geq \ell \) by \( \tau_t(\mathcal{F}) = \ell \). For any \( K_3 \in {F_1 \cap F_2 \brack \ell-1} \), \( K_3 \) is not a \( t \)-cover. Therefore, there exists some \( F_3 \in \mathcal{F} \) such that \( \dim(K_3 \cap F_3) \leq t-1 \). So
    \begin{align*}
      \dim(F_1 \cap F_2 \cap F_3) &=\dim(F_1 \cap F_2)+\dim(F_3)-\dim((F_1 \cap F_2)\cup F_3)\\
      &\leq \dim(F_1 \cap F_2)+\dim(F_3)-\dim(K\cup F_3)\\
      &=\dim(F_1 \cap F_2)+\dim(F_3)-(\dim(K)+\dim(F_3)-\dim(K\cup F_3))\\
      &=\dim(F_1 \cap F_2)-\dim(K)+\dim(K\cup F_3)\\
      &\leq s - (\ell - 1) + (t - 1) = s - (\ell - t).
    \end{align*}
    By the property of \( r \)-wise \( t \)-intersecting families, the dimension of the intersection of \( F_1 \cap F_2 \cap F_3 \) with any other space in \( \mathcal{F} \) is at least \( t \), implying that \( \dim(F_1 \cap F_2 \cap F_3) \geq \ell \). For any \( K_4 \in {F_1 \cap F_2 \cap F_3 \brack \ell-1} \), \( K_4 \) is not a \( t \)-cover. Hence, there exists \( F_4 \in \mathcal{F} \) such that \( \dim(K_4 \cap F_4) \leq t-1 \). This results in the inequality
    \[
    \dim(F_1 \cap F_2 \cap F_3 \cap F_4) \leq (s - (\ell - t)) - (\ell - 1) + (t - 1) = s - 2(\ell - t).
    \]
    Continuing this process, we obtain
    \[
    \dim\left( \bigcap_{i=1}^r F_i \right) \leq s - (r - 2)(\ell - t).
    \]

    Since \( \mathcal{F} \) is \( r \)-wise \( t \)-intersecting, we have that \( s \geq (r - 2)(\ell - t) + t \). Therefore, \( \mathcal{F} \) is \( ((r - 2)(\ell - t) + t) \)-intersecting, as required. This completes the proof, with the result following from the randomness of the choices of \( F_1 \) and \( F_2 \).
\end{proof}

\iffalse
\begin{lemma}{\rm (\cite[Lemma~2.5]{cao2023r})}\label{cao1}
	Let $n,$ $k,$ $\ell$ and $t$ be non-negative integers with $n\geq k+\ell$, and $\mathcal{F}\subseteq {V\brack k}$ and $\mathcal{G}\subseteq {V\brack \ell}$ be maximal cross $t$-intersecting families. Assume that $\tau_{t}(\mathcal{F})=m_k$  and $\tau_{t}(\mathcal{G})=m_\ell$. Then
	$$
	|\mathcal{F}|\leq{m_k\brack t}{n-t\brack k-t}.
	$$
	Moreover, the following hold.
	\begin{enumerate}[{\rm(i)}]
		\item If $m_\ell=t+1$, then
		$$
		|\mathcal{F}|\leq{m_k\brack t}{\ell-t+1\brack 1}{n-t-1\brack k-t-1}.
		$$
		\item If $m_\ell\geq t+2$, then
		$$
		|\mathcal{F}|\leq {m_k\brack t}{\ell\brack 1}^{m_\ell-t-2}{\ell-t+1\brack 1}^2{n-m_\ell\brack k-m_\ell}.
		$$
	\end{enumerate}
\end{lemma}
\fi

\begin{lemma}\label{cross}
    Let \( \mathcal{F}_1, \mathcal{F}_2, \ldots, \mathcal{F}_m\subseteq  {V\brack k} \) be \( k \)-uniform $m$-cross \( s \)-intersecting families, where \( \tau_t(\mathcal{F}_i) \geq \ell \) \( (\ell \geq t+1 )\) for \( 1 \leq i \leq m \), $n\geq 2k-s$. For each \( 1 \leq i \leq m \), we have
    \[
    |\mathcal{F}_i| \leq {k \brack \ell - 1} {s \brack \ell - 1}^{-1} {k \brack s} {n - k + t - \ell \brack k - s + t - \ell}.
    \]
\end{lemma}

\begin{proof}
    Without loss of generality, we bound the size of \( \mathcal{F}_1 \). Fix \( G \in \mathcal{F}_2 \). For every \(( \ell - 1 \))-subspace \( X \) of \( G \), we count all the \( F \in \mathcal{F}_1 \) so that  \( X\subseteq F \). Since \( \dim(F \cap G) \geq s \), each \( F \in \mathcal{F}_1 \) is counted at least \( {s \brack \ell - 1} \) times. Therefore, we have
    \[
    |\mathcal{F}_1| \leq {s \brack \ell - 1}^{-1} \sum_{X \in {G\brack l-1}} \left|\left
    \{ F \in \mathcal{F}_1 : X \in {F\brack\ell - 1} \right\} \right|.
    \]

    For any \( X \in {F\brack \ell - 1} \), there exists a \( G_X \in \mathcal{F}_2 \) such that \( \dim(G_X \cap X) \leq t - 1 \) by $\tau_t(\mathcal{F}_2) \geq \ell$. Now, for each \( F \in \mathcal{F}_1 \), we first choose an \( s \)-subspace \( H \) of \(  G_X \). From Lemma \ref{1}, for any fixed $H\in {  G_X \brack s}$, the number of \( F \in \mathcal{F}_1 \) containing \( X \) and \( H \) is bounded by
    \begin{align*}
        \left| \{ F \in \mathcal{F}_1 : X \subseteq F, H \subseteq F \} \right| &=\left| \{ F \in \mathcal{F}_1 : \text{span}\{X ,H\} \subseteq F \} \right| \\
        &\leq {n - k + t - \ell \brack k - s + t - \ell}.
    \end{align*}
    Here are ${k \brack s}$ choices of $H$. Thus, the total number of \( F \in \mathcal{F}_1 \) containing \( X \) is bounded by
    \begin{align*}
        \left| \{ F \in \mathcal{F}_1 : X \subseteq F \} \right| &\leq\sum_{H\in { G_X\brack s}}\left| \{ F \in \mathcal{F}_1 : X \subseteq F, H \subseteq F \} \right|\\
        &\leq {k \brack s} {n - k + t - \ell \brack k - s + t - \ell}.
    \end{align*}

    This completes the proof for the bound on \( |\mathcal{F}_1| \).
\end{proof}

\begin{lemma}\label{2tau}
    Let \( \mathcal{F} \subseteq {V \brack k} \) be an \( r \)-wise \(( r \geq 2 )\) \( t \)-intersecting family with \( \tau_t(\mathcal{F}) \geq t + 2 \) and \( n > 4k + 4kt/r \)  $($or $n\geq 2k+9$ for the case $r=2,t=1)$. Then,
    \[
    \mathcal{N}(\Delta_{r+1,t}, \mathcal{F}) < n_{t+r,k}.
    \]
\end{lemma}

\begin{proof}
    By Lemma \ref{rt}, \( \mathcal{F} \) is \( (2r + t - 4) \)-intersecting. Applying Lemma \ref{cross}, with $m=2$, \( \mathcal{F}_1 = \mathcal{F}_2 = \mathcal{F} \), \( s = 2r + t - 4 \), and \( \ell = t + 2 \), we get
    \[
    |\mathcal{F}| \leq {k \brack t + 1} {2r + t - 4 \brack t + 1}^{-1} {k \brack 2r + t - 4} {n - k - 2 \brack k - 2r - t + 2}.
    \]
   By Proposition \ref{eq} and Lemma \ref{n r+1,t}, we have
   \begin{align*}
	\log_q\mathcal{N}(\Delta_{r+1,t},\mathcal{F}) &\leq \log_q|\mathcal{F}|^{r+1} \\
	&\leq \log_q\Big({k\brack t+1}{2r+t-4\brack t+1}^{-1}{k\brack 2r+t-4}{n-k-2\brack k-2r-t+2}\Big)^{r+1}\\
        &< (r+1)((t+1)(k-t-1)-(t+1)(2r-5)\\
        &\quad+(2r+t-4)(k-2r-t+5)+(k-2r-t+2)(n-2k+2r+t-3))
        \end{align*}
        \begin{align*}
        &\leq (r+1)(k-r-t+1)n-(k+1)(r+1)(k-r-t+1)+r(k+1)-r^2\\
	&< \log_qn_{t+r,k}.
    \end{align*}
    The fourth inequality holds for $n\geq (-8r^3 + 7kr^2 + 18r^2 - 3rt^2 + 4krt - rk^2 - 2kr + 6rt + 3r - 3t^2 - 10tr^2 + 4kt + 16t - k^2 - 8k - 21)/(r^2-1)$. For simplicity, we can assume $n\geq\frac{4kr+4kt}{r}$. Specifically, when $r = 2$ and $t = 1$, the corresponding range is $n\geq\frac{-3k^2+28k-28}{3}$, which holds when $n \geq 2k+9$.
\end{proof}

If $\tau_t(\mathcal{F}) = t$, then $\mathcal{F}$ is a trivial $t$-intersecting family that does not contain any $(r+1,t)$-simplex. By Lemma \ref{2tau}, there cannot be too many $(r+1,t)$ -simplices when $\tau_t(\mathcal{F}) \geq t+2$. Thus, we only need to focus on the case where $\tau_t(\mathcal{F}) = t+1$.

\section{Proof of the main results}

 Let $\mathcal{F} \subseteq {V \brack k}$ be an $r$-wise $t$-intersecting family with $k \geq t \geq 1$ and $r \geq 2$. In this section, we assume that $\tau_t(\mathcal{F}) = t+1$. Our goal is to characterize $\mathcal{C}_t(\mathcal{F})$.

A subspace $X$ of $V$ is called \emph{cover-complete} (with respect to $\mathcal{F}$) if any $(t+1)$-subspace of $X$ is an element of $\mathcal{C}_t(\mathcal{F})$. It is easy to verify that the following proposition holds.

\begin{prop}\label{dim}
	Let $X$ be a cover-complete subspace. Then for any $F \subseteq \mathcal{F}$, we have $\dim(F \cap X) \geq \dim(X) - 1$.
\end{prop}

\begin{prop}\label{size}
	$|\mathcal{F}| \leq {k \brack t} {r+t-2 \brack t}^{-1} {k \brack r+t-2} {n-k-1 \brack k-r-t+1}$ for $n\geq 2k-r-t+2$.
\end{prop}
\begin{proof}
	By Lemma \ref{rt}, $\mathcal{F}$ is $(r+t-2)$-intersecting. In Lemma \ref{cross}, let $\mathcal{F}_1 = \mathcal{F}_2 = \mathcal{F}$, $s = r+t-2$ and $\ell = t+1$. Then we have $|\mathcal{F}| \leq {k \brack t} {r+t-2 \brack t}^{-1} {k \brack r+t-2} {n-k-1 \brack k-r-t+1}$.
\end{proof}
Since these two propositions are used frequently in the following proofs, we will not explicitly refer to them when using.

\begin{lemma}\label{P3}
	Let $X, Y \in \mathcal{C}_t(\mathcal{F})$ with $\dim(X \cap Y) = t$ and $n \geq 3kr^2 + 3krt$  $($or $n\geq 2k+9$ for the case $r=2,t=1)$. For any $(t+1)$-subspace $Z \subseteq \text{span}\{X, Y\}$, at least one of the following holds:
	\begin{enumerate}
		\item $Z \in \mathcal{C}_t(\mathcal{F})$ $($i.e., $\text{span}\{X, Y\}$ is cover-complete$)$.
		\item $\mathcal{N}(\Delta_t, \mathcal{F}) < n_{t+r,k}$.
	\end{enumerate}
\end{lemma}
\begin{proof}
	Since $\dim(\text{span}\{X, Y\}) = t+2$ and $\dim(X \cap Y) = t$, we have
	\begin{align*}
	    t \geq \dim(Z \cap X \cap Y) &\geq \dim(Z) + \dim(X \cap Y) - \dim(Z + (X \cap Y)) \\
        &\geq \dim(Z) + \dim(X \cap Y) - \dim(\text{span}\{X, Y\}) \\
        &= t-1.
	\end{align*}
    We now examine the possible values of $\dim(Z \cap X \cap Y)$.

\noindent\textbf{Case 1.} $\dim(Z \cap X \cap Y) = t$.
		
In this case, $X \cap Y \subseteq Z$. We claim that $Z \in \mathcal{C}_t(\mathcal{F})$.
			
Arbitrarily choose vectors $e_1 \in X \setminus (X \cap Y)$, $e_2 \in Y \setminus (X \cap Y)$, and $e_3 \in Z \setminus (X \cap Y)$. Then $X$, $Y$ and $Z$ can be written as $(X \cap Y) \oplus \langle e_1 \rangle$, $(X \cap Y) \oplus \langle e_2 \rangle$ and $(X \cap Y) \oplus \langle e_3 \rangle$, respectively. Hence $\text{span}\{X, Y\} = (X \cap Y) \oplus \langle e_1 \rangle \oplus \langle e_2 \rangle$. Since $e_3 \in Z \subseteq \text{span}\{X, Y\}$, we have $e_3 = ae_1 + be_2 + e^*$, where $e^* \in X \cap Y$ and $a\in \mathbb{F}_q\setminus\{0\}$ or $b \in \mathbb{F}_q\setminus\{0\}$. If $a=0$ or $b=0$, we have $Z=Y$ or $Z=X$, which shows that $Z \in \mathcal{C}_t(\mathcal{F})$. Therefore, we can assume that $a,b\in\mathbb{F}_q\setminus\{0\}$.
			
Suppose $Z \notin \mathcal{C}_t(\mathcal{F})$. Then there exists $F \in \mathcal{F}$ such that $\dim(F \cap Z) \leq t-1$. Consequently, $\dim(F \cap (X \cap Y)) \leq \dim(F \cap Z) \leq t-1$. Since $X, Y \in \mathcal{C}_t(\mathcal{F})$ and $\dim(F \cap (X \cap Y)) \geq t-1$, we must have $\dim(F \cap (X \cap Y)) = t-1$. Then there exist two vectors  $ce_1 + v_1\in (F\cap X) \setminus (X \cap Y)$ and $de_2 + v_2\in (F\cap Y) \setminus (X \cap Y)$, where $v_1, v_2 \in X \cap Y$ and $c, d \in \mathbb{F}_q \setminus \{0\}$. Thus, their linear combination $ac^{-1}(ce_1 + v_1) + bd^{-1}(de_2 + v_2) = e_3 + ac^{-1}v_1 + bd^{-1}v_2 \in F$, which is also in $Z \setminus (X \cap Y)$, implying $\dim(Z \cap F) \geq \dim(\text{span}\{F \cap (X \cap Y), e_3 + ac^{-1}v_1 + bd^{-1}v_2\}) \geq t$, a contradiction.

\noindent\textbf{Case 2.} $\dim(Z \cap X \cap Y) = t-1$.
				
Arbitrarily choose vectors $e_1 \in (X \cap Y )\setminus (X \cap Y \cap Z)$, $e_2 \in (X \cap Z) \setminus (X \cap Y \cap Z)$, and $e_3 \in (Y \cap Z) \setminus (X \cap Y \cap Z)$. Then $X$, $Y$, and $Z$ can be written as $(X \cap Y \cap Z) \oplus \langle e_1 \rangle \oplus \langle e_2 \rangle$, $(X \cap Y \cap Z) \oplus \langle e_1 \rangle \oplus \langle e_3 \rangle$ and $(X \cap Y \cap Z) \oplus \langle e_2 \rangle \oplus \langle e_3 \rangle$, respectively. Suppose $Z \notin \mathcal{C}_t(\mathcal{F})$. Then there exists $F \in \mathcal{F}$ such that $\dim(F \cap Z) \leq t-1$. Since $X,Y \in \mathcal{C}_t(\mathcal{F})$, we have $\dim(F \cap X \cap Y \cap Z) \geq t-2$.

\noindent\textbf{Subcase 2.1.} $\dim(F \cap X \cap Y \cap Z) = t-2$.

We will show the case can not occur indeed. We claim that $F\cap X\nsubseteq X\cap Z$. Otherwise, if $F\cap X\subseteq X\cap Z$, then we have $F\cap X= X\cap Z=F\cap X\cap Z$ by  $\dim(X\cap Z)= t$ and $\dim(F\cap X)\geq t$. Then  $\dim(F\cap Z)\geq\dim(F\cap X\cap Z)=\dim(F\cap X)\geq t$, a contradiction.

Therefore we can find a vector $v_1\in (F\cap X)\setminus (X\cap Z)$. Let $G$ be a subspace such that $F \cap X = (F \cap X \cap Y \cap Z) \oplus \langle v_1 \rangle \oplus G$. Since $\dim(F \cap X) \geq t$, we have $\dim(G) = \dim(F \cap X) - \dim((F \cap X \cap Y \cap Z) \oplus \langle v_1 \rangle) \geq 1$. There exist nonzero vector in $G$, denoted by $v_2$, such that $(F \cap X \cap Y \cap Z) \oplus \langle v_1 \rangle \oplus \langle v_2 \rangle \subseteq F \cap X$. Since $v_1,v_2\subseteq X\cap F$, and $v_1\notin X\cap Z$, we can assume $v_1= ae_1+be_2+v_3,v_2=ce_1+de_2+v_4$ where $v_3$, $v_4\in X\cap Y \cap Z$, $a\in \mathbb{F}_q\setminus\{0\},b,c,d \in \mathbb{F}_q$. Then $c(ae_1+be_2+v_3)-a(ce_1+de_2+v_4)=(bc-ad)e_2+cv_3-av_4\in F.$

We claim that $bc-ad\neq 0$. Otherwise $0=(cv_3-av_4)-c(ae_1+be_2+v_3)+a(ce_1+de_2+v_4)$. However, direct sum $(F\cap X\cap Y\cap Z)\oplus<v_1>\oplus<v_2>$ implies that $0$ can be only represented by $0+0+0$. While $a(ce_1+de_2+v_4)=av_2\neq 0$, a contradiction.
					
Using the same analysis for $F \cap Y$, we get $(b'c'-a'd')e_3+c'v'_3-a'v'_4\in F$ where $b'c'-a'd'\neq 0$. Finally,
$
	\dim(F\cap Z)\geq \dim(\text{span}\{F\cap X\cap Y\cap Z,(bc-ad)e_2+cv_3-av_4,(b'c'-a'd')e_3+c'v'_3-a'v'_4\})= t
$
by $bc-ad,b'c'-a'd'\neq 0$, a contradiction.
					
\noindent\textbf{Subcase 2.2.} $\dim(F\cap X\cap Y\cap Z)=t-1$.

                We claim that $X\cap Y \subseteq F$. Indeed, if $X\cap Y \nsubseteq F$, $t=\dim(X\cap Y)>\dim(F\cap X\cap Y)\geq \dim(F\cap X\cap Y\cap Z)=t-1$ which implies $\dim(F\cap X\cap Y)=t-1$. Since $X,Y\in \mathcal{C}_t(\mathcal{F})$, we can find  $v_1\in F\cap(X\setminus(X\cap Y))$, $v_2\in F\cap(Y\setminus(X\cap Y))$ and they can be expressed as $v_1=ae_1+be_2+v_3$ and $v_2=a'e_1+b'e_3+v_4$ where $v_3,v_4\in X\cap Y \cap Z$ and $a,a'\in \mathbb{F}_q$, $b,b'\in \mathbb{F}_q\setminus\{0\}$. Thus $a'v_1-av_2\in F$ and  $a'v_1-av_2\in Z\setminus(X\cap Y\cap Z)$, which indicates $\dim(Z\cap F)\geq \dim(\text{span}\{F\cap (X\cap Y\cap Z), a'v_1-av_2\})\geq t$, a contradiction.

                Since $\dim(F\cap X)= t$  (otherwise, by $X\in \mathcal{C}_t(\mathcal{F})$, we have $\dim(F\cap Z)\geq \dim(F\cap Z\cap X)\geq \dim(F\cap X)-1\geq t$, a contradiction), we have $X\cap F=X\cap Y$. Similarly, $X\cap F=Y\cap F=X\cap Y$. Let $\mathcal{F}'=\{F'\in \mathcal{F}|X\cap Y\nsubseteq F'\}$. Consider the cardinality of $\mathcal{F}'$. Since $X$, $Y\in \mathcal{C}_t(\mathcal{F})$ and $\dim(X\cap Y)=t$,  $\dim(F'\cap X) = \dim(F'\cap Y) =t$ for every $F'\in \mathcal{F}'$. So there exist vectors $x\in X\setminus (X\cap Y)$ and $y\in Y\setminus (X\cap Y)$ such that  $x,y\in F'$. From Lemma~\ref{rt}, $\dim(F\cap F')\geq r+t-2$.
                Since $X\cap F=Y\cap F=X\cap Y$, $x,y\notin F\cap F'$. Thus, for all $G\subseteq F\cap F'$ with $\dim(G)=r+t-2$ and $x\in X\setminus (X\cap Y),y\in Y\setminus (X\cap Y)$, we have $$\mathcal{F}'\subseteq \bigcup_{\substack{x\in X\setminus (X\cap Y),y\in Y\setminus (X\cap Y),G\in{F\cap F'\brack r+t-2}}}\left\{H\in{V\brack k}:\text{span}\{x,y,G\}\subseteq H\right\}$$ and
                \begin{align*}
                    \lvert\mathcal{F}'\rvert &\leq\sum_{G\in{F\cap F'\brack r+t-2}} \sum_{x\in X\setminus (X\cap Y)} \sum_{y\in Y\setminus (X\cap Y)}\left|\left\{H\in{V\brack k}:\text{span}\{x,y,G\}\subseteq H\right\}\right|\\
                    &\leq {k\brack r+t-2}\Big(q^t\Big)^2{n-r-t\brack k-r-t}.
                \end{align*}
                Observe that any $(r+1,t)$-simplex must contain at least one subspace in $\mathcal{F}'$. Thus
				\begin{align*}
					\log_q\mathcal{N}(\Delta_t,\mathcal{F}) &\leq\log_q(|\mathcal{F}'||\mathcal{F}|^r)\\
					&\leq\log_q\left({k\brack r+t-2}q^{2t}{n-r-t\brack k-r-t}\right)\\
                    &\quad+\log_q\left(\Big({k\brack t}{r+t-2\brack t}^{-1}{k\brack r+t-2}{n-k-1\brack k-r-t+1}\Big)^r\right).
                    \end{align*}
                   By Proposition \ref{eq}, we have
                    \begin{align*}
                                    \log_q\mathcal{N}(\Delta_t,\mathcal{F}) &\leq (r+t-2)(k-r-t+3)+2t+(k-r-t)(n-k+1)\\
                        &\quad+r(t(k-t+1)-t(r-2)+(r+t-2)(k-r-t+1))\\
            &\quad+r((k-r-t+1)(n-2k+r+t-1))\\
                        &\leq (r+1)(k-r-t+1)n-(k+1)(r+1)(k-r-t+1)+r(k+1)-r^2\\
                        &<\log_qn_{t+r,k}.
				\end{align*}
                The second-to-last inequality holds for $n\geq -2r^3+4r^2-3kr+kt-rk^2+5rt+k-t^2+5t-3rt^2-5r^2t+3kr^2+4krt-5$. For simplicity, we can assume $n \geq 3kr^2 + 3krt$. Specifically, when $r = 2$ and $t = 1$, the corresponding range is $n\geq -2k^2+15k-17$, which holds when $n \geq 2k+9$.
\end{proof}

\begin{lemma}\label{clique}
	Let $X$ be an $m$-dimensional cover-complete subspace of $V$ $($with $m \geq t+1)$ with respect to $\mathcal{F}$, $Y \in \mathcal{C}_t(\mathcal{F})$ with $\dim(X \cap Y) = t$, and $n \geq 3kr^2 + 3krt$  $($or $n\geq 2k+9$ for the case $r=2,t=1)$. For any $(t+1)$-subspace $Z \subseteq \text{span}\{X, Y\}$, at least one of the following holds:
	\begin{enumerate}
		\item $Z \in \mathcal{C}_t(\mathcal{F})$ $($i.e., $\text{span}\{X, Y\}$ is cover-complete$)$.
		\item $\mathcal{N}(\Delta_t, \mathcal{F}) < n_{t+r, k}$.
	\end{enumerate}
\end{lemma}

\begin{proof}
	Assume $\mathcal{N}(\Delta_t, \mathcal{F}) \geq n_{t+r, k}$. We now aim to prove that $Z \in \mathcal{C}_t(\mathcal{F})$ holds for all $(t+1)$-subspaces of $\text{span}\{X, Y\}$. Each time when we apply Lemma \ref{P3}, we can infer that some $(t+1)$-subspaces of $\text{span}\{X, Y\}$ are in $\mathcal{C}_t(\mathcal{F})$. By repeatedly applying  Lemma \ref{P3}, we eventually conclude that all $(t+1)$-subspaces of $\text{span}\{X, Y\}$ are in $\mathcal{C}_t(\mathcal{F})$ and then complete the proof. A more detailed proof is as follows.

{\noindent\bf Claim.} Let $e$ be a vector in $\text{span}\{X, Y\} \setminus X$ and $U$ be a $t$-subspace of $X$. If $\text{span}\{U, e\} \in \mathcal{C}_t(\mathcal{F})$, then $\text{span}\{W, e\} \in \mathcal{C}_t(\mathcal{F})$ for any $W \in {X \brack t}$.

 {\noindent\bf Proof of Claim.} Given $W \in {X \brack t}$ and $W\neq U$, let the basis of $U\cap W$ be $\{e_1,e_2,\ldots,e_a\}$, where $0\leq a\leq t-1$ and $a=0$ means $\dim(U\cap W)=0$. Assume $\{e_1,e_2,\ldots e_a,e_{a+1},\ldots, e_t\}$ and $\{e_1,e_2,\ldots,e_a,e'_{a+1},e'_{a+2},\ldots, e'_t\}$ are the bases of $U$ and $W$
 respectively. Let $Z_i=\text{span}\{e_1,e_2,\ldots,e_{t-i},e'_{t-i+1},e'_{t-i+2},\ldots, e'_t\}$ for $0\leq i\leq t-a$ (in fact, $Z_0=U$ and $Z_{t-a}=W$). Noticing that $\{e_1,e_2,\ldots,e_{t},e'_{a+1},e'_{a+2},\ldots, e'_t\}$ are linearly independent, we have $\dim(U \cap Z_1) = \dim(Z_i \cap Z_{i+1}) = \dim(W \cap Z_{t-a-1}) = t-1$ for any $0\leq i\leq t-a-1$. Since $X$ is an $m$-dimensional cover-complete subspace, we have $\text{span}\{U, Z_1\}\in \mathcal{C}_t(\mathcal{F})$ and $\text{span}\{U, e\}\in \mathcal{C}_t(\mathcal{F})$. By Lemma \ref{P3}, we conclude that $\text{span}\{U, Z_1, e\}$ is cover-complete. Thus, $\text{span}\{Z_1, e\} \in \mathcal{C}_t(\mathcal{F})$. Using the same analysis, since $\text{span}\{Z_1,Z_2\}$ and $\text{span}\{Z_1, e\}$ are in $\mathcal{C}_t(\mathcal{F})$, by Lemma \ref{P3}, we conclude that $\text{span}\{Z_2, e\}$ is cover-complete. Thus, $\text{span}\{Z_2, e\} \in \mathcal{C}_t(\mathcal{F})$. By repeating this process, we obtain that $\text{span}\{Z_i, e\}$ for any $0\leq i\leq t-a$, which concludes the proof of this claim.\q

It remains to show that $Z \in \mathcal{C}_t(\mathcal{F})$ for any $Z \in { \text{span}\{X, Y\} \brack t+1}$.

If $Z\in {X\brack t+1}$, the result holds by $X$ being cover-complete. Hence, we assume $\dim(Z\cap X)=t$. Notice that $Z\in \mathcal{C}_t(\mathcal{F})$ if $Z\subseteq X$ or $Z=Y$. Therefore we only need to focus on the case that $Z\nsubseteq X$ and $Z\neq Y$. Let $v_1 \in Y \setminus (X \cap Y)$ and $v_2 \in Z \setminus (X \cap Z)$. Let $\{e_1, e_2, \dots, e_t\}$ be a basis of $X \cap Y$, which can be extended to a basis of $X$, denoted by $\{e_1, e_2, \dots, e_m\}$. Then $v_2$ can be written as $v_2 = v_1 + \sum_{i=1}^m a_i e_i$, where $ a_i \in \mathbb{F}_q$. Since $\text{span}\{e_1, e_2, \dots, e_t, v_1\} = Y \in \mathcal{C}_t(\mathcal{F})$, and $v_1 + \sum_{i=1}^t a_i e_i \in \text{span}\{e_1, e_2, \dots, e_t, v_1\},$ we have
$$\text{span}\left\{e_1, e_2, \dots, e_{t-1}, e_{t+1}, v_1 + \sum_{i=1}^t a_i e_i\right\} \in \mathcal{C}_t(\mathcal{F})$$
by the claim above (where $U=\text{span}\{e_1, e_2, \dots, e_t\}$, $W=\text{span}\{e_1, e_2, \dots, e_{t-1}, e_{t+1}\}$ and $e=v_1 +\sum_{i=1}^t a_i e_i$).

Using the same analysis and $$\text{span}\{e_1, e_2, \dots, e_{t-1}, e_{t+1}, v_1 + \sum_{i=1}^{t+1} a_i e_i\}=\text{span}\{e_1, e_2, \dots, e_t, v_1 + \sum_{i=1}^{t} a_i e_i\} \in \mathcal{C}_t(\mathcal{F}),$$ we have
$$\text{span}\left\{e_1, e_2, \dots, e_{t-1}, e_{t+2}, v_1 + \sum_{i=1}^{t+1} a_i e_i\right\} \in \mathcal{C}_t(\mathcal{F})$$
by the claim above (where $U=\text{span}\{e_1, e_2, \dots, e_{t-1}, e_{t+1}\}$, $W=\text{span}\{e_1, e_2, \dots, e_{t-1}, e_{t+2}\}$ and $e=v_1 +\sum_{i=1}^{t+1} a_i e_i$). Repeating this argument, we obtain
$$\text{span}\{e_1, e_2, \dots, e_{t-1}, e_m, v_2\}=\text{span}\left\{e_1, e_2, \dots, e_{t-1}, e_m, v_1 + \sum_{i=1}^{m-1} a_i e_i\right\} \in \mathcal{C}_t(\mathcal{F}).$$
Thus, by the claim above (set $U=\text{span}\{e_1, e_2, \dots, e_{t-1}, e_{m}\}$, $W=Z\cap X$ and $e=e_2$), we conclude that $Z \in \mathcal{C}_t(\mathcal{F})$.
\end{proof}

We call a subspace $X$ \textit{maximal cover-complete} if any cover-complete subspace containing $X$ must be equal to $X$. Lemma \ref{clique} actually describes the conditions under which a cover-complete subspace $X$ cannot be maximal. Next, we use the following two lemmas to prove that the maximal cover-complete subspace is of dimension $r+t$.

\begin{lemma}\label{ntbig}
	Let $X$ be a cover-complete subspace with $\dim(X) = m \geq t + r + 1$, and $n \geq 3kr^2 + 3krt$  $($or $n\geq 2k+9$ for the case $r=2,t=1)$. Then, $\mathcal{N}(\Delta_t, \mathcal{F}) < n_{t+r, k}$.
\end{lemma}

\begin{proof}
	Let $Y \in {X \brack t}$. Define $\mathcal{F}' = \{ F \in \mathcal{F} \mid Y \not\subseteq F \}$. The intersection of $Y$ and any $F \in \mathcal{F}'$ is a $(t-1)$-subspace (otherwise, any $(t+1)$-subspace of $X$ containing $Y$ cannot cover $F$, which leads to a contradiction).

	Consider the cardinality of $\mathcal{F}'$. Fixing $F \cap Y$, let the basis of $F \cap Y$ be $\{e_1, e_2, \dots, e_{t-1}\}$, and extend it to the basis of $Y$ and $X$ as $\{e_1, e_2, \dots, e_t\}$ and $\{e_1, e_2, \dots, e_m\}$, respectively. Since $\text{span}\{e_1, e_2, \dots, e_t, e_i\} \in \mathcal{C}_t(\mathcal{F})$, we can find a vector $e_i + v_i \in F$ where $v_i \in Y$ for every $t+1 \leq i \leq m$. There are ${t \brack t-1}$ possible choices for $F \cap Y$, and $q^t$ distinct vectors $e_i + v_i$.
	Thus we have
	\(
	|\mathcal{F}'| \leq {t \brack t-1} \Big(q^t\Big)^{m-t} {n-m+1 \brack k-m+1}.
	\)
	Next, we estimate $\mathcal{N}(\Delta_t, \mathcal{F})$.
	\begin{align*}
		\log_q\mathcal{N}(\Delta_t, \mathcal{F}) &\leq \log_q(|\mathcal{F}'| |\mathcal{F}|^r) \\
		&\leq \log_q\left({t \brack t-1} q^{t(m-t)} {n-m+1 \brack k-m+1}\right) \\
            &\quad+\log_q\left( {k \brack t} {r+t-2 \brack t}^{-1} {k \brack r+t-2} {n-k-1 \brack k-r-t+1} \right)^r.
            \end{align*}
            By Proposition \ref{eq}, we have
            \begin{align*}
            \log_q\mathcal{N}(\Delta_t, \mathcal{F})&\leq2(t-1)+t(m-t)+(k-m+1)(n-k+1)\\
            &\quad+r(t(k-t+1)-t(r-2)+(r+t-2)(k-r-t+1))\\
            &\quad+r((k-r-t+1)(n-2k+r+t-1))\\
            &\leq 2(t-1)+t(r+1)+(k-t-r)(n-k+1)\\
            &\quad+r(t(k-t+1)-t(r-2)+(r+t-2)(k-r-t+1))\\
            &\quad+r((k-r-t+1)(n-2k+r+t-1))\\
            &\leq (r+1)(k-r-t+1)n-(k+1)(r+1)(k-r-t+1)+r(k+1)-r^2\\
		&<\log_qn_{t+r, k}.
	\end{align*}
        The second-to-last inequality holds for $n\geq 2t+8tr-4kr-4r-k^2r+3kr^2-3t^2r-5tr^2+4ktr+2k-2r^3+5r^2-1$. For simplicity, we can assume $n \geq 3kr^2 + 3krt$. Specifically, when $r = 2$ and $t = 1$, the corresponding range is $n\geq-2k^2+14k-8$, which holds when $n \geq 2k+9$.
\end{proof}
\begin{lemma}\label{ntsmall}
	Let $X$ be a maximal cover-complete subspace with $ \dim(X) = m \leq t + r - 1$, and $n \geq 3kr^2 + 3krt$  $($or $n\geq 2k+9$ for the case $r=2,t=1)$. Then $\mathcal{N}(\Delta_t, \mathcal{F}) < n_{t+r, k}$.
\end{lemma}

\begin{proof} Since $X$ is  cover-complete, $m\ge t+1$.
	 Fix a $X'$ such that $X\oplus X'=V$, and let $(F\cap X,F(X'))$ be the projection of $F\in \mathcal{F}$ onto $(X,X')$. For  $Y \in {X \brack t}$, let $\mathcal{F}_Y = \{ F \in \mathcal{F} : Y \not\subseteq F \}$ and  $\mathcal{F}'_Y = \{F(X'): F\in\mathcal{F}_Y\}$. Then $\mathcal{F}_Y$ is a $t$-intersecting family and $\mathcal{F}'_Y$ is $(k-m+1)$-uniform.

\noindent\textbf{Case 1.} There is $Y \in {X \brack t}$ such that $\tau_1(\mathcal{F}'_Y)=1$.

		Fix a $F \in \mathcal{F}_Y$. Let the basis of $Y$ be $\{e_1, e_2, \dots, e_t\}$, and extend it to the basis of $X$ as $\{e_1, e_2, \dots, e_m\}$. Notice that $\text{span}\{Y, e_i\} \in \mathcal{C}_t(\mathcal{F})$ for all $t + 1 \leq i \leq m$ by $X$ being  cover-complete. Then $\dim(Y \cap F) = t - 1$ and there exists $e_i + v_i \in F$, where $v_i \in Y$ for $t + 1 \leq i \leq m$. Since $\tau_1(\mathcal{F}'_Y)=1$, there exists $x\in X'\setminus\{0\}$ such that $x \in G$ for all $G\in \mathcal{F}'_Y$. Then, there exists a vector $x + v + \sum_{i = t + 1}^m a_i e_i$ where $v \in Y$ and $a_i \in \mathbb{F}_q$. Thus, $$x + v + \sum_{i = t + 1}^m a_i e_i - a_i (e_i + v_i) = x + v - \sum_{i = t + 1}^m a_i v_i \in F.$$
We have $\dim(F \cap \text{span}\{x, Y\}) \geq \dim(\text{span}\{F \cap Y, x + v - \sum_{i = t + 1}^m a_i v_i\}) = t$. Since $x$ is independent of the choice of $F$, we conclude that $\text{span}\{x, Y\} \in \mathcal{C}_t(\mathcal{F})$. By Lemma \ref{clique}, either $\text{span}\{X, x\}$ is cover-complete, or $\mathcal{N}(\Delta_t, \mathcal{F}) < n_{t+r, k}$. Since $X$ is maximal, we conclude that $\mathcal{N}(\Delta_t, \mathcal{F}) < n_{t+r, k}$.

\noindent\textbf{Case 2.} $\tau_1(\mathcal{F}'_Y)\ge 2$  for any $Y \in {X \brack t}$.

        Given  $Z \in {X \brack m-1}$, let $\mathcal{H}_Z = \{ H \in \mathcal{F} : X \cap H = Z \}$ and $\mathcal{H}'_Z = \{H(X'):H\in\mathcal{H}_Z\}$. Choose $Y\in {Z \brack t}$ and let $F'\in \mathcal{F}_Y$. Noticing that $\dim(Z\cap F'\cap X)=m-2$. By Lemma \ref{rt}, $\mathcal{F}$ is $(r+t-2)$-intersecting. So we have  $\dim(H(X')\cap F'(X'))\geq r+t-m\geq 1$ for any $H\in \mathcal{H}_Z$. Thus there must be a $1$-subspace in $H(X')\cap F'(X')$ for any $H\in \mathcal{H}_Z$. For any $1$-subspace $<v>$ of $F'(X')$, there exists $F_{<v>}\in \mathcal{F}_Y$ such that $v\notin F'_{<v>}(X')$ by $\tau_1(\mathcal{F}'_Y)\ge 2$. We have $\dim(H(X')\cap F'_{<v>}(X'))\geq r+t-m$ for any $H\in \mathcal{H}_Z$. Hence
        \begin{align*}
        	|\mathcal{H}'_Z|&\leq \left|\bigcup_{<v>\in {F'(X')\brack 1}}\left\{J\in {X'\brack k-m+1}:<v>\in J,\dim(J\cap F'_{<v>}(X'))\geq r+t-m\right\}\right|\\
        	&\leq {k-m+1\brack 1}{k-m+1\brack r+t-m}{n-r-t-1\brack k-r-t}.
        \end{align*}
        	Since $\mathcal{F}_Y \subseteq \bigcup\limits_{Z} \mathcal{H}_Z$, we obtain
        \[
        |\mathcal{F}_Y| \leq \sum\limits_{Z} |\mathcal{H}_Z| \leq {m \brack m - 1} {k-m+1\brack 1}{k-m+1\brack r+t-m}{n-r-t-1\brack k-r-t}.
        \]
        Any $(r+1,t)$-simplex in $\mathcal{F}$ must contain at least one element in $\mathcal{F}_Y$. Therefore, by Proposition~\ref{eq}, we have
        \begin{align*}
        	\log_\mathcal{N}(\Delta_t, \mathcal{F}) &\leq \log_q\left(|\mathcal{F}_Y| |\mathcal{F}|^r\right) \\
        	&\leq \log_q\left({m \brack m - 1} {k-m+1\brack 1}{k-m+1\brack r+t-m}{n-r-t-1\brack k-r-t}\right) \\
        	&\quad+\log_q\left(\Big({k\brack t}{r+t-2\brack t}^{-1}{k\brack r+t-2}{n-k-1\brack k-r-t+1}\Big)^r\right)\\
        	&\leq2(m-1)+(k-m+1)+(r+t-m)(k-r-t+2)+(k-r-t)(n-k)\\
        	&\quad+r(t(k-t+1)-t(r-2)+(r+t-2)(k-r-t+1))\\
        	&\quad+r((k-r-t+1)(n-2k+r+t-1))\\
        	&\leq 2t+(k-t)+(r-1)(k-r-t+2)+(k-r-t)(n-k)\\
        	&\quad+r(t(k-t+1)-t(r-2)+(r+t-2)(k-r-t+1))\\
        	&\quad+r((k-r-t+1)(n-2k+r+t-1))\\
        	&\leq (r+1)(k-r-t+1)n-(k+1)(r+1)(k-r-t+1)+r(k+1)-r^2\\
        	&<\log_qn_{t+r, k}.
        \end{align*}
        The second-to-last inequality holds for $n\geq -2r^3+3kr^2+4r^2-3rt^2+4krt+6rt-rk^2-3kr-r-5tr^2+t+2k-1$. For simplicity, we can assume $n \geq 3kr^2 + 3krt$.
        Specifically, when $r = 2$ and $t = 1$, the corresponding range is $n\geq-3k^2+21k-22$, which holds when $n \geq 2k+9$.
\end{proof}
By Lemmas \ref{ntbig} and  \ref{ntsmall},  we obtained that unless
any maximal cover-complete subspace of $\mathcal{F}$ is $(r + t)$-dimensional, the number of $(r+1,t)$-simplices of $\mathcal{F}$ is strictly smaller than
$ n_{t + r, k}$. Let $X$ be an $(r + t)$-dimensional maximal cover-complete subspace of $\mathcal{F}$. Then we have $\mathcal{F} \subseteq \mathcal{F}_{X, k}$ by Proposition \ref{dim}. We claim that the number of $ (r+1,t)$-simplices in $\mathcal{F}^*_{X, k}$ equals to its in $\mathcal{F}_{X, k}$. Indeed,  for any $(r+1,t)$-simplex $\{F_1,F_2,\ldots,F_{r+1}\}\subseteq \mathcal{F}_{X, k}$, we have $\dim(\bigcap_{i=1}^{r+1}(X\cap F_i))\leq \dim(\bigcap_{i=1}^{r+1}F_i)< t$. Suppose $\{F_1,F_2,\ldots,F_{r+1}\}\nsubseteq \mathcal{F}^*_{X, k}$,  say $F_1\notin \mathcal{F}^*_{X, k}$. Then $X\cap F_1=X$. Noticing that $\dim(F_i\cap X)\geq m-1$, we have $\dim(\bigcap_{i=2}^{r+1}(X\cap F_i))\geq m-r=t$. Hence, $\dim(\bigcap_{i=1}^{r+1}(X\cap F_i))=\dim(\bigcap_{i=2}^{r+1}(X\cap F_i))\geq m-r=t$, a contradiction. Therefore, for any simplex $(r+1,t)$-simplex $\{F_1,F_2,\ldots,F_{r+1}\}\subseteq \mathcal{F}_{X, k}$, we have $\{F_1,F_2,\ldots,F_{r+1}\}\subseteq \mathcal{F}^*_{X, k}$, which completes the proof of the claim. Moreover, deleting any element in $\mathcal{F}^*_{X, k}$, say $F$, strictly decreases the number of $ (r+1,t)$-simplices (there exists $\{F,F_2,F_3\ldots,F_{r+1}\}\subseteq \mathcal{F}^*_{X, k}$ satisfying the condition in \textbf{Step 4} of Lemma \ref{n r+1,t} which implies $F$ is in some simplices in $\mathcal{F}^*_{X, k}$). Thus, $\mathcal{F}^*_{X, k}\subseteq \mathcal{F}\subseteq \mathcal{F}_{X, k}$, which completes the proof of Theorem \ref{th} and Corollary \ref{cor}.
\iffalse
In the proof process, every time we get $\mathcal{N}(\Delta_t, \mathcal{F}) < n_{t + r, k}$, we can actually conclude that $\mathcal{N}(\Delta_t, \mathcal{F}) = o(n_{t + r, k})$. We can then obtain the following "stability" conclusion.

\begin{cor}
	Let $\mathcal{F} \subseteq {V \brack k}$ be an $r$-wise $t$-intersecting family, where $k \geq t \geq 1$ and $r \geq 2$. For any positive number $0 < c < 1$, if $\mathcal{N}(\Delta_{r+1,t}, \mathcal{F}) \geq c n_{t + r, k}$, then for sufficiently large $n$, $\mathcal{F} \subseteq \mathcal{F}_{X, k}$ for some $r + t$-subspace $X$ of $V$.
\end{cor}
\fi
\section*{Acknowledgement}
J. Song is supported by the Beijing Natural Science Foundation's Undergraduate Initiating Research Program (QY24222). M. Cao is supported by the National Natural Science Foundation of China (12301431),the Fundamental Research Funds for the Central Universities, and the Research Funds of Renmin University of China (24XNYJ11). M. Lu is supported by the National Natural Science Foundation of China (12171272).

%%%%%%%%%%%%%%%%%%%%%%%%	
\vskip.2cm
%{\bf Acknowledgments}
%Many thanks to the anonymous referee for his/her many helpful comments and suggestions, which
%have considerably improved the presentation of the paper.

%\end{CJK*}

\end{document}